\documentclass[11pt,letterpaper]{article}
\usepackage[letterpaper, margin=0.75in]{geometry}
\usepackage[letterpaper]{geometry}
\usepackage{graphicx,color}
\usepackage[cp850]{inputenc}
\usepackage[T1]{fontenc}
\usepackage[english]{babel}
\usepackage{rotating}
\usepackage[f]{esvect}
\usepackage{amsmath, amssymb}
\usepackage{amsthm}
\usepackage{rotating}
\usepackage{lscape}
\usepackage{mathrsfs}
\usepackage{mathtools}
\usepackage[toc,page]{appendix}
\usepackage{hyphenat}

\usepackage[normalem]{ulem}

\def\tablenotes{\bgroup\parfillskip=0pt plus 1fil
\leftskip=0pt\relax \rightskip=0pt
\vskip2pt\footnotesize}
\def\endtablenotes{\vskip1pt\egroup}

\newtheorem{theorem}{Theorem}[section]
\newtheorem{proposition}[theorem]{Proposition}

\newtheorem{lemma}[theorem]{Lemma}

\newcommand{\TimeDeriv}{\frac{\textrm{d}}{\textrm{dt}}}

\renewcommand{\leq}{\leqslant}
\renewcommand{\geq}{\geqslant}
\renewcommand{\d}{\mathrm{d}}

\usepackage{amsopn}

\renewcommand{\phi}{\varphi}
\renewcommand{\epsilon}{\varepsilon}
 
\allowdisplaybreaks

 \usepackage[square,sort,numbers]{natbib}

\usepackage{graphicx}
\usepackage{overpic}
\usepackage{tikz}
\numberwithin{equation}{section}

\makeatletter
\renewcommand{\@biblabel}[1]{#1\hfill \hspace{-0.1cm}}
\makeatother

\title{Distributed Delay Differential Equation Representations of Cyclic Differential Equations \thanks{Submitted to the editors DATE.
\funding{Portion of this work were performed under the auspices of the U.S. Department of Energy under contract 89233218CNA000001 and funded by NIH grants R01-AI116868 and R01-OD011095. }}}
\begin{document}
 \title{Distributed Delay Differential Equation Representations of Cyclic Differential Equations }
 \author{ Tyler Cassidy  \\
          Theoretical Biology and Biophysics, Los Alamos National Laboratory \\     
          Los Alamos, NM, USA  87545\\  
           tcassidy@lanl.gov  }
\maketitle

\begin{abstract}
{Compartmental ordinary differential equation (ODE) models are used extensively in mathematical biology. When transit between compartments occurs at a constant rate, the well-known linear chain trick can be used to show that the ODE model is equivalent to an Erlang distributed delay differential equation (DDE). Here, we demonstrate that compartmental models with non-linear transit rates and possibly delayed arguments are also equivalent to a scalar distributed delay differential equation. To illustrate the utility of these equivalences, we calculate the equilibria of the scalar DDE, and compute the characteristic function-- without calculating a determinant. We derive the equivalent scalar DDE for two examples of models in mathematical biology and use the DDE formulation to identify physiological processes that were otherwise hidden by the compartmental structure of the ODE model. }
\end{abstract}

\section{Introduction}\label{Sec:Introduction}

Multi-compartment models,  where changes in one population propagate through a chain of successive stages, have been used extensively in mathematical biology. Examples include inhibitory (and excitatory) neuronal feedback loops \citep{Eurich2002,Mackey1984,Foss2000,Milton2011}, cellular reproduction \citep{deSouza2017,Roskos2006,Quartino2014,Billy2014,Cassidy2018}, enzymatic production \citep{Goodwin1965,Yildirim2004,Alrikaby2019}, infectious disease epidemiology \citep{Kermack1927,Mckendrick1925,Inaba2008,Champredon2018} and many others. It is well established that, when the relationship between stages is linear, these compartmental models ``hide'' delays \citep{Vogel1961,deSouza2017,Smith2011,Gurney1986,Cassidy2018a}. Recently, there has been increased interest establishing the equivalence between models that explicitly include delays, like renewal or distributed delay differential equations (DDEs), and multi-stage ordinary differential equation (ODE) models \citep{Champredon2018,Diekmann2017,Diekmann2020,Diekmann2020a,Hurtado2019}.  

In general, these multi-stage models follow a chain-like structure, with one population influencing the next. When there is feedback between the first and last populations, these chain-like structures close and   become cyclic. Here, we formalize the relationship between these cyclic differential equations and distributed DDEs. Specifically, we establish the equivalence between a scalar distributed DDE and the general, possibly delayed, cyclic differential equation
\begin{equation}
\TimeDeriv x_i(t) = f_i\left( \int_0^{\infty} x_{i-1}(t-\phi)K_{i}(\phi)\d \phi \right) - \left( \mu(x_n(t))  \right)   x_i(t), \quad \textrm{for} \quad i = 1,...n,
\label{Eq:GenericCyclicDifferentialEquation}
\end{equation}
and the indices $i$ are taken mod $n$. Equation~\eqref{Eq:GenericCyclicDifferentialEquation} includes the integral term
\begin{align*}
\int_0^{\infty} x_{i}(t-\phi)K_{i}(\phi)\d \phi ,
\end{align*} 
where each $K_{i}(\phi)$ is a probability density function (PDF). Thus, we study the relationship between scalar distributed DDEs and multi-stage models that potentially include a delay. In particular, two specific formulations of \eqref{Eq:GenericCyclicDifferentialEquation} have been extensively studied. First, by choosing $K_{i}(\phi) = \delta(\phi-\tau_{i})$, \eqref{Eq:GenericCyclicDifferentialEquation} becomes a system of cyclic discrete DDEs with delays given by $\tau_{i}$ given by
\begin{equation}
\TimeDeriv x_i(t)  = f_i\left( x_{i-1}(t-\tau_{i}) \right) - \left( \mu(x_n(t)) \right)   x_i(t) \quad \textrm{for} \quad i = 1,...n,
\label{Eq:DiscreteDelayCyclicDifferentialEquation}
\end{equation}
where, once again, the indices $i$ are taken mod $n$.
The system~\eqref{Eq:DiscreteDelayCyclicDifferentialEquation} has been studied in depth by a number of authors \citep{Braverman2019,Ivanov2018,MalletParet1996}. Theoretical results include a Poincar\'{e}-Bendixson theorem for the discrete system of DDEs \eqref{Eq:DiscreteDelayCyclicDifferentialEquation} when $\mu(s) = 0$ \citep{MalletParet1996}, and the existence of periodic solutions of \eqref{Eq:DiscreteDelayCyclicDifferentialEquation} under modest assumptions on the specific feedback functions $f_i$ \citep{Braverman2019,Ivanov2018}.  We consider a particular example of \eqref{Eq:DiscreteDelayCyclicDifferentialEquation}, used in the context of lac-operon dynamics \citep{Yildirim2004}, in Section~\ref{Sec:Examples}.

Conversely, \eqref{Eq:GenericCyclicDifferentialEquation} is quite common in mathematical modelling in the delay free case: after setting $K_{i}(\phi) = \delta(\phi),$ the delay in \eqref{Eq:GenericCyclicDifferentialEquation} vanishes, and the system becomes a multi-compartment ODE. Then, the equivalence of an Erlang, or gamma type distribution with an integer shape parameter,   distributed DDE and a system of ODEs has been known since at least the 1960s \citep{Vogel1961}. The linear chain trick, or linear chain technique (LCT), establishes the equivalence between Erlang distributed DDEs and transit compartment ODE models with constant transition rate \citep{Smith2011,MacDonald1978}. Recently, a number of authors have generalized the LCT to other distributions and model formulations \citep{Diekmann2020,Diekmann2020a,Hurtado2019}. Often, these transit compartment ODE models take the form
\begin{equation}
\left \{
\begin{aligned}
\TimeDeriv x_1(t) & =   \frac{ \beta( x_n(t) ) }{V}-V x_1(t) \\ 
\TimeDeriv x_i(t) & =  V \left[ x_{i-1}(t) - x_{i}(t)\right]  \quad \textrm{for} \quad i = 2,3,...,n-1, \\
\TimeDeriv x_n(t) & = F(x_n(t),V x_{n-1}(t) )  \\
\end{aligned}
\right \}
\label{Eq:ConstantRateTransitODE}
\end{equation}
where, for the constant transit rate between compartments $V$, we see $ f_i(x_{i-1}(t) ) = V x_{i-1}(t)$ and $ \mu(x_n(t))x_i(t) = V x_i(t),$ while $\beta(x_n(t))$ is the recruitment rate into the chain of transit compartments. The LCT consists of replacing the transit compartment chain $\{x_i(t)\}_{i=1}^{n-1}$ with the distributed delay term 
\begin{align}
x_{n-1} (t) = \int_0^{\infty} \beta(x_n(t-s)) g_{V}^{n-1}(s) \d s
\label{Eq:ConstantRateDelayTerm}
\end{align}
where $g_{V}^{n-1}(s)$ is the PDF of the gamma distribution with scale parameter $V$ and shape parameter $n-1$
\begin{equation*}
g_{V}^{n-1}(s) = \frac{V^{n-1}s^{n-2}e^{-Vs}}{(n-2)!}. 
\end{equation*}
The linchpin of the LCT is the ability to write $g_{V}^i(s)$ as the solution of a system of differential equations
\begin{equation*}
\frac{\d}{\d s} g_{V}^1(s) = -V g_{V}^1(s) \quad \textrm{and} \quad \frac{\d}{\d s} g_{V}^i(s) = V[g_{V}^{i-1}(s)-g_{V}^i(s)], 
\end{equation*}
which is an explicit example of a sufficient condition to replace a distributed DDE by a system of ODEs \citep{Vogel1965,Fargue1974}, namely that the delay kernel $K(t)$ must satisfy 
\begin{equation*}
\frac{\d^n}{\d t^n} K(t)+ \displaystyle \sum_{i=0}^{n-1} a_i(t) \frac{\d^i}{\d t^i} K(t) = 0.
\end{equation*}

Often, particularly in the pharmaceutical sciences, the transit rate and clearance terms are not constant, but rather determined through an external variable, $y(t)$, so $ f_i(x_{i-1}(t),y(t) ) =  V(y(t))x_{i-1}(t)$ and $\mu(x_n(t),y(t)) = V(y(t))x_i(t)$ \citep{Krzyzanski2011,Hu2018,Roskos2006,deSouza2017}. Naively including a variable transit rate, $V(y(t))$ in \eqref{Eq:ConstantRateTransitODE} gives
\begin{equation}
\left \{
\begin{aligned}
\TimeDeriv x_1(t) & =   \frac{ \beta x(t) }{V(y(t))} - V(y(t))x_1(t)  \\[0.2cm]
\TimeDeriv x_i(t) & =  V(y(t))\left[ x_{i-1}(t) - x_{i}(t)\right] \quad \textrm{for} \quad i = 2,3,...,n-1 \\ [0.2cm]
\TimeDeriv x_n(t) & = F(x_n(t),V(y(t))x_{n-1}(t) )- \gamma(x_n(t)) x_n(t). \\
\end{aligned}
\right \}
\label{Eq:GammaEquivalentODE}
\end{equation}

Cassidy et al. \citep{Cassidy2018a} established the equivalence between \eqref{Eq:GammaEquivalentODE} and a state dependent gamma distributed DDE by explicitly considering the age-structured PDE modelling the underlying maturation process. In the variable transit rate case, the distributed delay term \eqref{Eq:ConstantRateDelayTerm} becomes, for $j = 1,2,...,n-1$, 
\begin{equation*}
 x_{n-1}(t) =  \int_{0}^{\infty} g_{V}^j \left( \int_{t-\phi}^t \frac{V(y(s))}{V}  \d s\right) \frac{ \beta x(t-\phi) }{V(y(t-\phi))}  \d \phi.
\end{equation*}
While the results developed in this work translate to models that include external control, we do not focus on state dependent distributed DDEs.

The model ingredients necessary to derive equations such as \eqref{Eq:ConstantRateTransitODE} or \eqref{Eq:GammaEquivalentODE} were considered in \citep{Gurney1986,Diekmann2020,Diekmann2020a}. Broadly speaking, creating a model like \eqref{Eq:ConstantRateTransitODE} or \eqref{Eq:GammaEquivalentODE} requires determining the birth (or appearance) rate $\beta(x_n(t))$, the death (or growth rate) $\mu(x_n(t))$ and the ageing (or transit rate) $V(y(t))$. These model ingredients are precisely those catalogued by Diekmann and collaborators in their work on physiologically structured equations \citep{Diekmann2020,Diekmann2020a}. In brief, these model ingredients allow for the development of a physiologically structured model. In their recent work, Diekmann and coworkers derived necessary and sufficient criteria to determine if the, typically infinite dimensional structured models, can be reduced to a finite dimensional system of ODEs without the loss of relevant information \citep{Diekmann2017,Diekmann2020,Diekmann2020a}. 

The physiologically structured models considered by Diekmann and collaborators offer a framework to study the role of individual level heterogeneity on population level dynamics. These structured models allow for individuals to be continuously distributed in ``trait'' (i.e age, size, maturity,\ldots)  space, rather than imposing the artificial binning that would be necessary in the ODE case. In general, these structured population models describe the evolution of a density $p$ over the set of possible ``traits'', which provide the physiological structure, $\Omega$. Often, the population distribution across the possible states determines the model output and is a density over $\Omega$, so $p \in L_1(\Omega)$.  It is then natural to consider the population level dynamics, given by the time evolution of 
\begin{equation*}
N(t) = \int_{\Omega} \psi(x) n(t,x)\d x .
\end{equation*}
The function $\psi$ acts as a weight function in the mapping the distribution of individual states to the  population, equivalently the mapping $L_1(\Omega) \to \mathbb{R}^k$. Through careful bookkeeping, it is sometimes possible to cast the evolution of $N(t)$ as a delay, or renewal, equation  \citep{Diekmann1998,Diekmann2010}
\begin{equation*}
N(t) = F(N_t),
\end{equation*}
where $N_t = N(t+\theta), \  \theta \in (-\infty,0)$ and, for $\rho > 0$, solutions live in the natural phase space   \citep{Diekmann2012}
\begin{align*}
L_{1,\rho} = \left \{ f \ \bigg| \int_{-\infty}^0 |f(\phi)| e^{\rho \phi} \d \phi < \infty \right \}, 
\end{align*}

Here, we employ a similar book keeping strategy when considering the cyclic system~\eqref{Eq:GenericCyclicDifferentialEquation} to obtain a scalar distributed DDE. Effectively, by tracking the appearance or recruitment rate into each compartment and measuring the expansion or contraction of each cohort, we write down a component-wise solution of the transit stages in the cyclic differential equation~\eqref{Eq:GenericCyclicDifferentialEquation} in Section~\ref{Sec:GeneralizedLinearChainTechnique}. Then, similar to the LCT, we are left with a scalar distributed DDE. However, unlike the classical LCT and existing variants, our technique extends to models with both nonlinear clearance rates and the delayed terms from \eqref{Eq:GenericCharacteristicEquation}. We then show how recasting the system of $n$ DDEs as the equivalent scalar distributed DDE simplifies model analysis by establishing non-negativity of solutions, giving an explicit expression for equilibria and calculating the characteristic equation by making extensive use of the chain rule for Fr\'{e}chet derivatives to replace the $n \times n$ determinant typically involved the calculation of the characteristic function in Section~\ref{Sec:ScalarDDEProperties}. Next, we consider two biological systems and corresponding mathematical models which take the form \eqref{Eq:GenericCyclicDifferentialEquation} in Section~\ref{Sec:Examples}. In particular, these examples elucidate how the chain-like structure of \eqref{Eq:GenericCyclicDifferentialEquation} hide delayed processes that are crucial in the physiological system, and offer the opportunity to illustrate the general theory established in the preceding sections while demonstrating how the equivalence between a cyclic differential equation and a scalar distributed DDE can be implemented in practice. We finish with a discussion of the mathematical and biological advantages of our work in a brief conclusion.

\section{Generalized linear chain technique}\label{Sec:GeneralizedLinearChainTechnique}

In this section, we demonstrate how to reduce \eqref{Eq:GenericCyclicDifferentialEquation} to a scalar distributed DDE. As mentioned, the theory for scalar DDEs is quite well studied, so this reduction enables simpler analysis of the equivalent system. We note that the case with no explicit delays has been extensively studied and catalogued by Diekmann et al. \citep{Diekmann2017,Diekmann2020,Diekmann2020a}. For ease of notation, we separate our analysis into two cases: the first with only one explicit delay in \eqref{Eq:GenericCyclicDifferentialEquation} and the second with multiple explicit delays. In what follows, we use $x_{i,t}$ to denote the function segment $x_{i,t}(\theta) = x_i(t+\theta)$ for $ \theta \in (-\infty,0).$
 In the first case, to avoid cumbersome notation, we consider a specific case of \eqref{Eq:GenericCyclicDifferentialEquation} with  $n=3$
\begin{equation}
\left \{
\begin{aligned}
\TimeDeriv x_1(t)&  = f_1\left( \int_0^{\infty} x_3(t-\phi)K_{1}(\phi)\d \phi \right) - \left( \mu(x_3(t))  ) \right)   x_1(t) \\ 
\TimeDeriv x_2(t)&  = f_2 \left( x_1(t) \right) - \left( \mu(x_3(t)) \right)  x_2(t) \\
\TimeDeriv x_3(t)&  = f_3\left( x_2(t) \right) - \left( \mu(x_3(t))  \right)  x_3(t) \\
\end{aligned}
\right \}
\label{Eq:n3SingleDelayCyclicDifferentialEquation}
\end{equation}

We note that the differential equation for $x_1(t)$ in \eqref{Eq:n3SingleDelayCyclicDifferentialEquation} is linear in $x_1$ and, otherwise, is a possibly non-linear function of $x_3$. Specifically, the term $f_1\left( \int_0^{\infty} x_3(t-\phi)K_{1}(\phi)\d \phi \right)$, which is independent of $x_1(t)$, can be thought of as the recruitment rate at time $t$, while the factor $ \mu(x_3(t)) $  gives the growth or contraction rate of $x_1(t)$ at time $t$. Then, using Leibniz's rule, it is possible to verify that
\begin{align*}
x_1(t) & = \int_{0}^{\infty} f_1\left( \int_0^{\infty} x_3(t-s-\phi)K_{1}(\phi)\d \phi \right)\exp\left(- \int_{t-s}^t \mu(x_3(u)) \d u \right) \d s \\
 & = \int_{-\infty}^{t} f_1\left( \int_0^{\infty} x_3(\sigma-\phi)K_{1}(\phi)\d \phi \right)\exp\left(- \int_{\sigma}^t \mu(x_3(u)) \d u \right) \d \sigma. 
\end{align*}
We note that $x_1(t)$ is entirely determined by $x_3(t)$ and that the expression 
\begin{equation*}
f_1\left( \int_0^{\infty} x_3(t-s-\phi_1)K_{1}(\phi_1)\d \phi_1 \right)\exp\left(- \int_{t-s}^t \mu(x_3(u)) \d u \right),
\end{equation*}
is the product of the recruitment into $x_1$ at time $t-s$ and the expansion or contraction, determined by the sign of $\mu$, of that cohort between time $t-s$ and $t$. Using the same technique, we obtain
\begin{align*}
x_2(t) & = \int_0^{\infty} f_2 \left( x_1(t-r) \right) \exp\left( - \int_{t-r}^t \mu(x_3(u))  \d u \right) \d r  \\
&  = \int_{-\infty}^{t} f_2 \left( x_1(r) \right) \exp\left( - \int_{r}^t \mu(x_3(u)) \d u \right) \d r.
\end{align*}
Now, using the expression for $x_1(t)$, we see that 
\begin{align*}
x_2(t) & = \int_{-\infty}^{t} f_2 \left[ \int_{-\infty}^{r} f_1\left( \int_0^{\infty} x_3(\sigma-\phi_1 )K_{1}(\phi_1)\d \phi_1 \right) \exp\left(- \int_{\sigma}^r \mu(x_3(u))  \d u \right) \d \sigma  \right]  \\
& \quad {} \times \exp\left( - \int_{r}^t \mu(x_3(u))  \d u \right) \d r.
\end{align*} 
Once again, we note that $x_2(t)$ is entirely determined by $x_3(t)$ alone, so we finally obtain the scalar distributed DDE
\begin{align*} 
\TimeDeriv x_3(t) & = f_3 \left( x_2(t) \right) - \left( \mu(x_3(t))  \right)  x_3(t) \\   
& = f_3 \left(  \int_{-\infty}^{t} f_2 \left[ \int_{-\infty}^{r} f_1 \left( \int_0^{\infty} x_3(\sigma-\phi_1)K_{1}(\phi_1) \d \phi_1 \right) \exp\left(- \int_{\sigma}^r  \mu(x_3(u))  \d u \right) \d \sigma  \right] \right. \\ 
& \quad {} \left. \times \exp \left( - \int_{r}^{t} \mu(x_3(u)) \d u \right) \d r \right) - \left( \mu(x_3(t)) \right)  x_3(t). 
\end{align*}

We begin formalizing the relationship between the chain structure of \eqref{Eq:GenericCyclicDifferentialEquation} and a scalar distributed DDE by partially solving the differential equations for the transit compartments.

\begin{lemma}\label{Lemma:ChainReduction}
Assume that $[x_1(t), x_2(t),...,x_n(t)]$ solves \eqref{Eq:GenericCyclicDifferentialEquation}. Then $x_i(t) = F_i(x_{i-1,t},x_{n,t})$ for $i \geq 2$.
\end{lemma}
\begin{proof}
The proof follows the structure of the previous example where $n=3$, with the $i = 1$ case following verbatim with
\begin{align*}
x_1(t) & = \int_{0}^{\infty} f_1\left( \int_0^{\infty} x_n(t-s-\phi)K_{1}(\phi)\d \phi \right)\exp\left(- \int_{t-s}^t \mu(x_n(u))  \d u \right) \d s = F_1(x_{n,t}).
\end{align*}
Now, consider 
\begin{equation*}
\TimeDeriv x_{i+1}(t) = f_{i+1}\left( \int_0^{\infty} x_{i}(t-\phi)K_{i+1}(\phi)\d \phi \right) - \left( \mu(x_n(t)) \right)   x_{i+1}(t),
\end{equation*}
and note that the above differential equation is linear in $x_{i+1}$ and potentially non-linear in $x_i$ and $x_n$. Using the same strategy as in the $n=3$ example, we see that 
\begin{align*}
x_{i+1}(t) & = \int_0^{\infty} f_{i+1} \left[ \int_0^{\infty}  x_{i}(t-s-\phi)K_{i+1}(\phi)\d \phi \right] \exp\left( - \int_{t-s}^t \mu(x_n(u))  \d u \right) \d s = F_i(x_{i,t},x_{n,t}),
\end{align*} 
which completes the claim. 
\end{proof}

It then follows that, as in the LCT, we can close the cycle by writing $x_i(t) = F_i(x_{n,t} )$ for $i = 1,2,...,n-1 $.  Thus, the dynamics of \eqref{Eq:GenericCyclicDifferentialEquation} are determined by the dynamics of $x_n$.
\begin{theorem}\label{Theorem:DistributedDDEReduction}
Let $[x_1(t),x_2(t),...,x_n(t)$ satisfy \eqref{Eq:GenericCyclicDifferentialEquation}. Then, $x_n(t)$ satisfies a scalar distributed DDE.
\end{theorem}
\begin{proof}
Using Lemma~\ref{Lemma:ChainReduction}, we write
\begin{equation*}
x_1(t) = F_1(x_{n,t}), \quad \textrm{and} \quad x_i(t) = F_i(x_{i-1,t},x_{n,t}).
\end{equation*}
Then, as $x_1(t)= F_1(x_{n,t})$, it follows that $x_2(t) = F_2(F_1(x_{n,t}),x_{n,t}) = G_2(x_{n,t})$. Now, we can repeat this for $i = 3,...,n-1$, and obtain 
\begin{equation*}
x_{n-1}(t) = F_{n-1}( x_{n-2,t},x_{n,t} )  = F_{n-1}( F_{n-2} ( x_{n-3,t},x_{n,t} ), x_{n,t} )= ... =   G_{n-1}(x_{n,t}),
\end{equation*}
so that
\begin{align*}
\TimeDeriv x_n(t) &  = f_n\left( \int_0^{\infty} x_{n-1}(t-\phi)K_{n}(\phi)\d \phi \right) - \left( \mu(x_n(t)) \right)   x_n(t) \\
 &  = f_n\left( \int_0^{\infty} G_{n-1}(x_n(t-\phi)) K_{n}(\phi)\d \phi \right) - \left( \mu(x_n(t)) \right)   x_n(t).
\end{align*} 

\end{proof}

To complete the equivalence between the scalar distributed DDE and the system of cyclic differential equations~\eqref{Eq:GenericCyclicDifferentialEquation}, we must map the initial data from one formulation to the other. This can be slightly complicated, as the dimensions of phase space may be different in each formulation. For example, the classic LCT establishes the equivalence between an Erlang distributed DDE with initial data in the infinite dimensional probability space given by the Erlang PDF  with a system of ODEs with finite dimensional phase space.

\begin{theorem}\label{Theorem:n3ReductionToDistributedDDE}
The cyclic differential equation \eqref{Eq:GenericCyclicDifferentialEquation}  and initial data given by
\begin{equation*}
x_i(s) = \xi_i(s) \quad \textrm{for} \quad s \in (-\infty,t_0],
\end{equation*}
is equivalent to a scalar distributed delay differential equation for $x_n(t)$ given by 
\begin{equation}
\TimeDeriv x_n(t) =  f_n\left( \int_0^{\infty} G_{n-1}(x_n(t-\phi)) K_{n}(\phi)\d \phi \right) - \left( \mu(x_n(t)) \right)   x_n(t).
\label{Eq:ReducedDistributedDDE}
\end{equation}
for $G_{n-1}(x_n(t))$ given by Theorem~\ref{Theorem:DistributedDDEReduction} and a suitably chosen initial function $\psi(s)$.
\end{theorem}

\begin{proof}

From Theorem~\ref{Theorem:DistributedDDEReduction}, we can write $x_i(t)= F_i(x_n(t))$ for $i = 2,3,\ldots, n-1.$ Thus, the dynamics of \eqref{Eq:GenericCyclicDifferentialEquation} are completely determined by
\begin{equation*}
\TimeDeriv x_n(t) =  f_n\left( \int_0^{\infty} G_{n-1}(x_n(t-\phi)) K_{n}(\phi)\d \phi \right) - \left( \mu(x_n(t))  \right)   x_n(t).
\end{equation*}
To show equivalence between the two fromulations, \eqref{Eq:GenericCyclicDifferentialEquation} and \eqref{Eq:ReducedDistributedDDE}, we must show that, for given $ \{\xi_i(s)\}_{i=1}^n$, it is possible to construct a suitable $\psi(s)$ of the scalar distributed DDE and vice versa.

Assume that $\psi(s)$ is given, so setting 
\begin{equation}
\xi_i(s) = G_i(\psi(s)) \quad \forall s \in (-\infty,t_0]
\label{Eq:InitialDataCondition}
\end{equation} 
gives appropriate initial conditions for \eqref{Eq:GenericCyclicDifferentialEquation}

Now, assume that $ \{\xi_i(s)\}_{i=1}^n$ are given, and we must construct a history function $\psi(s)$ for the scalar distributed DDE. The function $\psi$ must be such that \eqref{Eq:InitialDataCondition} holds. Necessarily, we must have $\psi(s) = \xi_n(s)$  $K_{n}$-almost everywhere in $(-\infty,t_0)$. This imposes constraints on the remaining $\xi_i(s)$, as $\xi_i$ must simultaneously satisfy  
\begin{equation*}
\xi_i(s) = F_i(\xi_n(s)) \quad K_{i} \textrm{-almost everywhere}.
\end{equation*} 
We note that the equality only must hold $K_{i}$ almost-everywhere, which is equivalent to the history functions being equal in the equivalence class $L_1(K_{i})$. In the case that $k_{i} = \delta(t)$, i.e. the no delay case, then Cassidy and Humphries \citep{Cassidy2018} demonstrate how to construct a suitable history function.
\end{proof}

In general, a system of DDEs like \eqref{Eq:GenericCyclicDifferentialEquation} takes initial data in the infinite dimensional phase space \citep{Diekmann2012}
\begin{align*}
C_{0,\rho,} = \left \{ f \in C_0 \ \bigg| \lim_{\phi \to -\infty} f(\phi)e^{\rho \phi} = 0 \right \}   
\end{align*}
As the phase space of the cyclic differential formulation and the scalar distributed DDE are both infinite dimensional, the strict condition on the history functions $\xi_i$ in the preceding equivalence is perhaps unsurprising. Conversely,  the phase space of a compartmental ODE model is $\mathbb{R}^n$, so there is more ``space'' to exploit when constructing an appropriate history function for the LCT.

\section{Properties of the scalar distributed DDE}\label{Sec:ScalarDDEProperties}

Equation~\eqref{Eq:GenericCyclicDifferentialEquation} has been extensively studied in both the discrete delay case, where $K_{i}(s) = \delta(s-\tau_{i})$ and the no delay case where $K_{i}(s) = \delta(s)$ \citep{Ivanov2018,Braverman2019,MalletParet1996}.  As we are primarily interested in biological systems demonstrating a cyclic nature, we begin by demonstrating that, for modest assumptions on the functions $f_i$, solutions of  \eqref{Eq:GenericCyclicDifferentialEquation} evolving from non-negative initial data remain non-negative.
\begin{proposition}
Assume that $\mu $ is bounded above so $\mu(x_n) \leq \mu_{max}$ and that the initial data $\xi_n$ satisfies
\begin{equation*}
\int_{-\infty}^{0} \xi_n(0-\phi)K_{i}(\phi) \d \phi > 0  \quad \textrm{for} \quad i = 1,2,...,n.
\end{equation*} 
Further, assume that each $f_i$ satisfies
\begin{equation*}
f_{i}(x) > 0 \quad \textrm{if} \quad  x > 0 \quad \textrm{and} \quad f_i(0) = 0 \quad i = 1,2,...,n.
\end{equation*}
Then, the solution of the IVP \eqref{Eq:ReducedDistributedDDE} satisfies $x_n(t) \geq 0$ for all $t > 0$.
\end{proposition}
\begin{proof}
To begin, we note that if $G_{n-1}(x_{n,t}) \geq 0 $ $K_{n}$- almost everywhere, then 
\begin{equation*}
\TimeDeriv x_n(t) \geq - \mu(x_n(t))x_n(t) \geq -\mu_{max}x_n(t)
\end{equation*}
and Gronwall's inequality gives
\begin{equation*}
x_n(t) \geq x_n(0) \exp\left[ - \mu_{max}t \right] \geq 0. 
\end{equation*}
Therefore, to establish the claim, it is sufficient to show $G_{n-1}(x_{n,t}) \geq 0.$  From  
\begin{align*}
G_i(x_{n,t}) & = F_i(G_{i-1}(x_{n,t}),x_{n,t}) \\
& = \int_0^{\infty} f_{i} \left[ \int_0^{\infty}  G_{i-1}(x_{n,t-s-\phi}) K_{i}(\phi)\d \phi \right] \exp\left( - \int_{t-s}^t \mu(x_n(u))  \d u \right) \d s,
\end{align*}
and the assumption on $f_i$, if $G_{i-1} \geq 0 $-- $K_{i}$ almost-everywhere, then $ G_i(x_{n,t}) \geq 0 .$ 

Now, consider  
\begin{align*} 
G_1(x_{n,t}) = F_1(x_{n,t}) & = \int_{0}^{\infty} f_1\left( \int_0^{\infty} x_n(0-s-\phi)K_{1}(\phi)\d \phi \right)\exp\left(- \int_{-s}^0 \mu(x_n(u))  \d u \right) \d s, 
\end{align*}
and note that if $x_{n,t} \geq 0 $, then $G_1(x_{n,t}) \geq 0$. We consider two distinct cases.

\textbf{Case I.} Assume that $\xi_n(0)> 0,$ and let $t^*$ be the first time such that $x_n(t^*) = 0$. Then, for $s \in [0,t^*]$, $x_n(s) \geq 0$ and we obtain $G_i(x_{n,s}) \geq 0$. Then, for $t \in [0,t^*]$, we have
\begin{equation}
\TimeDeriv x_n(t) \geq -\mu_{max} x_n(t), 
\end{equation}
and Gronwall's inequality gives 
\begin{align*}
0 = x(t^*) \geq \xi_n(0)\exp(-\mu_{max}t^*) > 0,
\end{align*} 
which is a contradiction so no $t^*$ can exist. 

\textbf{Case II.} Assume that $\xi_n(0) =0$. Now, if $G_{n-1} = 0$-$K_n$ almost-everywhere for all $t > 0$, then $x_n = 0$ is the solution of the differential equation. Alternatively, let $\hat{t}$ be the first time such that 
\begin{align*}
\int_0^{\infty} G_{n-1}(x_{n,\hat{t}-\phi})K_{n}(\phi) d \phi > 0,
\end{align*}
so 
\begin{align*}
\left. \TimeDeriv x_n(t) \right|_{t = \hat{t}} =  f_n\left( \int_0^{\infty} G_{n-1}(x_{n,\hat{t}-\phi})K_{n}(\phi) d \phi  \right) > 0,
\end{align*}
so $x_n$ becomes positive at time $\hat{t}$ and we return to \textbf{Case I.}

\end{proof}

After establishing a mathematical model, a first step is often the study of equilibria. In \eqref{Eq:GenericCyclicDifferentialEquation}, an equilibrium solution is a vector of constant functions $[x_1^*,x_2^*,...,x_n^*]$ such that 
\begin{equation*}
\left[ \TimeDeriv x_1(t),\TimeDeriv x_2(t),...,\TimeDeriv x_n(t) \right] = [0,0,...,0].
\end{equation*}
Consequently, calculating the equilibrium solution involves simultaneously finding the zeros of $n$ nonlinear multivariate functions, which is slightly simplified by the form of \eqref{Eq:GenericCyclicDifferentialEquation} despite the nonlinearities. Conversely, equilibria $x^*$ of \eqref{Eq:ReducedDistributedDDE} satisfy the single variable equation
\begin{equation}
0 = f_n\left(G_{n-1}(x^*) \right)  - \mu(x^*) x^*.
\label{Eq:EquilibriumReducedCondition}
\end{equation}
 In the case of \eqref{Eq:EquilibriumReducedCondition}, we can use techniques from single variable calculus to establish existence and uniqueness of an equilibrium solution.  Defining $\mu^* = \mu(x^*)$ and returning to the definition of $G_i(x_n))$, we calculate
\begin{align*}
G_1(x^*) & =  \int_{0}^{\infty} f_1\left( \int_0^{\infty} x^* K_{1}(\phi)\d \phi \right)\exp\left(-  \mu^* s \right) \d s = \frac{f_1(x^*)}{\mu^*},  
\end{align*}
and 
\begin{align*}
G_i(x^*) & = \int_0^{\infty} f_{i} \left[ \int_0^{\infty}  G_{i-1}(x^*) K_{i}(\phi)\d \phi \right] \exp\left( - \mu^* s \right) \d s  \frac{ f_{i}(G_{i-1}(x^*)) }{\mu^*} = \frac{f_i}{\mu^*} \circ \frac{f_{i-1}}{\mu^*} \circ \ldots \circ \frac{f_1(x^*)}{\mu^*}. 
\end{align*}
We note that $ f_i(x_{i-1}^*)) /\mu^*$ is precisely the term that would be obtained by solving \eqref{Eq:GenericCyclicDifferentialEquation} for the $n$ different components of an equilibrium solution.

\subsection{Characteristic function of the scalar distributed DDE}

Once an equilibrium solution has been found, often the next step is to study the local stability of the equilibrium. As shown by Diekmann and Gyllenberg \citep{Diekmann2012}, the local stability of an equilibrium $x^*$ is determined via the position of zeros of the characteristic function. For systems of $n$ DDEs given by
\begin{align*}
\TimeDeriv y(t) = F(y,y_t),
\end{align*}
the characteristic function is determined by solving a transcendental eigenvalue problem arising from the $n \times n$ determinant
\begin{align*}
\det \left[ \lambda I - A - \mathcal{L}[B](\lambda) \right], 
\end{align*}
where $A$ and $B$ are the Fr\'{e}chet derivatives of $F$ with respect to $y$ and $y_t$ evaluated at the equilibrium point $y^*$. We now demonstrate how the reduced scalar distributed DDE can simplify the calculation of the characteristic equation. Assume that $x^*$ solves \eqref{Eq:EquilibriumReducedCondition}, so 
\begin{align*}
f_n\left(G_{n-1}(x^*) \right) = \mu^* x^*,
\end{align*}
and define $z(t) = x(t)- x^*$ with
\begin{align} \notag
\TimeDeriv z(t) &  = f_n\left(\int_0^{\infty} G_{n-1}(x(t-\phi))K_n(\phi) \d \phi \right) - \mu(x(t)) x(t) \\
& = f_n\left(\int_0^{\infty} G_{n-1}(x^*+z(t-\phi))\d \phi \right) - \mu(x^*+z(t)) (x^*+z(t)). 
\label{Eq:ZDifferentialEquation}
\end{align}

To complete the linearisation, we first consider non-delayed arguments of the right hand side of \eqref{Eq:ZDifferentialEquation} with linear approximation
\begin{equation}
 \mu(x^*+z(t)) (x^*+z(t)) = \mu^*x^* + \mu^*z(t) + \mu'(x^*) z(t) x^* + \mathcal{O}(z^2).
\end{equation}
We now turn to the delayed argument in \eqref{Eq:ZDifferentialEquation}, and must compute the Fr\'{e}chet derivative of the operator $H$ that maps $ \psi \in C_{0,\rho} $
\begin{align*}
H: \psi \to f_n\left(\int_0^{\infty} G_{n-1}( \psi(t-\phi))k_n(\phi) \d \phi \right).
\end{align*}
The chain rule for Fr\'{e}chet derivatives evaluated at the equilibrium $x^*$ gives
\begin{align*}
DH = f_n'(G_{n-1}(x^*) ) DG_{n-1} \psi = f_n'(x^*) \circ DG_{n-1} \circ DG_{n-2} \circ \ldots \circ DG_{1} \psi
\end{align*}
Now, we compute
\begin{align*}
DG_1 \psi  & = \int_{0}^{\infty}  f_1'(x^*) \left[ \int_0^{\infty} \psi(t-s-\phi)K_{1}(\phi)\d \phi \right]   e^{- \mu^*s } \d s \\
& \quad {} + \int_{0}^{\infty} e^{- \mu^*s }  f_1 (x^*) \left[ \int_{t-s}^t \mu'(x^*) \psi(x) \d x \right]  \d s \\ 
\end{align*}
and after setting $ \psi(t)  = e^{\lambda t}$, we get
\begin{align*}
DG_1 \psi  & =  \mathcal{L}[K_{1}](\lambda) \mathcal{L}[f_1'(x^*)](\mu^*+\lambda) e^{\lambda t} + \mu'(x^*) f_1 (x^*) \left[ \int_0^{\infty} e^{- \mu^*s } \left( \frac{e^{\lambda t} - e^{\lambda(t-s)} }{\lambda} \right) \d s \right] \\  
& {} = \mathcal{L}[K_{1}](\lambda) \mathcal{L}[f_1'(x^*)](\mu^*+\lambda) e^{\lambda t} + \mu'(x^*) f_1  (x^*) \left( \frac{1}{\mu^*} - \frac{1}{\mu^* + \lambda } \right) \frac{ e^{\lambda t}}{\lambda}  \\
& {} = \left( \mathcal{L}[K_{1}](\lambda) \mathcal{L}[f_1'(x^*)](\mu^*+\lambda) + \mathcal{L}\left[ \frac{ \mu'(x^*) f_1 (x^*)}{\mu^* } \right] \left( \mu^* + \lambda \right)  \right) \psi  
\end{align*}

As the above calculation holds for $i = 2,3,...n-1$, it follows from induction that 
\begin{align*}
DH \psi  & =  f_n'(G_{n-1}(x^*) ) \prod_{i=1}^{n-1} \left( \mathcal{L}[K_{i}](\lambda) \mathcal{L}[f_i'(x^*)](\mu^*+\lambda) + \mathcal{L}\left[ \frac{ \mu'(x^*) f_i (x^*)}{\mu^* } \right] \left( \mu^* + \lambda \right)  \right) \psi, 
\end{align*}
where $\psi(t) = Ce^{\lambda t}$. Then, $z(t) = x(t)- x^*$ satisfies the linear differential equation
\begin{align*}
\TimeDeriv z(t)& = DHz - \left[  \mu^*z(t) + \mu'(x^*) z(t) x^*  \right]
\end{align*}
which, using the ansatz $z = Ce^{\lambda t}$ and the resulting expression for $DH$, becomes
\begin{align*}
\lambda z(t) & = f_n'(G_{n-1}(x^*) ) \prod_{i=1}^{n-1} \left( \mathcal{L}[K_{i}](\lambda) \mathcal{L}[f_i'(x^*)](\mu^*+\lambda) + \mathcal{L}\left[ \frac{ \mu'(x^*) f_i (x^*)}{\mu^* } \right] \left( \mu^* + \lambda \right)  \right)z(t) \\
&{} \quad  - \left[ \mu^*+ \mu'(x^*) x^*  \right]z(t) .
\end{align*} 

Cancelling the $z(t)$ terms gives the characteristic equation
\begin{align} \notag
\lambda & = f_n'(G_{n-1}(x^*) ) \prod_{i=1}^{n-1} \left( \mathcal{L}[K_{i}](\lambda) \mathcal{L}[f_i'(x^*)](\mu^*+\lambda) + \mathcal{L}\left[ \frac{ \mu'(x^*) f_i (x^*)}{\mu^* } \right] \left( \mu^* + \lambda \right)  \right) \\
&{} \quad  - \left[ \mu^*+ \mu'(x^*) x^*  \right] . 
\label{Eq:GenericCharacteristicEquation}
 \end{align}
 
While these computations are cumbersome due to the notation involved, if we were to add an additional stage to \eqref{Eq:GenericCyclicDifferentialEquation}, updating the characteristic equation \eqref{Eq:GenericCharacteristicEquation} would be straightforward in this formulation. In particular, we would avoid calculating an $(n+1) \times (n+1)$ determinant, and simply have one extra factor in the multiplication. In Section~\ref{Sec:Examples}, we illustrate the simplicity of calculating the characteristic equation of the scalar distributed DDE for equations arising in biological modelling. 

In general, expanding the product of Laplace transforms yields $n$ different convolutions. In many biological examples, the growth or clearance rate is not state dependent, so $\mu'(x^*) = 0$ and the product of Laplace transforms becomes 
\begin{multline*}
\prod_{i=1}^{n-1} \mathcal{L}[K_{i}](\lambda) \mathcal{L}[f_i'(x^*)](\mu^*+\lambda)   \\
= \mathcal{L}[K_1*K_2* \ldots *K_{n-1}](\lambda) \mathcal{L}[f'_1(x^*) * f'_2(x^*)*\dots * f'_{n-1}(x^*)](\lambda+\mu^* ). 
\end{multline*}

Interestingly, the convolution of the PDFs $K_i$ represent the concatenation of the delayed process wherein changes in $x_1$ propagate to $x_n$ in the cyclic differential equation formulation given by \eqref{Eq:GenericCyclicDifferentialEquation}. As the densities $K_i(\phi)$ are only defined for $\phi >0$, the convolution of Laplace transforms is the moment generating function for the random variable modelling the time delay between the first and the $n$-th compartment. As the sojourn times in each stage are independent, this random variable is the sum of the random variables defining the sojourn time in each stage. Consequently, the mean delay between the first and $n$-th compartment is precisely the sum of the mean sojourn times in each compartment, as would be expected. Moreover, this form of the characteristic equation emphasizes the  concatenation of delayed processes modelled by the system of cyclic differential equations \eqref{Eq:GenericCyclicDifferentialEquation}. For completeness, we note that this term is present in the more general case where $\mu'(x^*) \neq 0$.

\section{Examples}\label{Sec:Examples}
The form of \eqref{Eq:GenericCyclicDifferentialEquation} is quite general and encompasses a large number of mathematical models of physiological processes, including those discussed earlier. Here, we consider models of two distinct biological processes to illustrate the general technique derived in Section \ref{Sec:GeneralizedLinearChainTechnique}. We begin with a model of the dynamics of the lac-operon, in which sequential expression of intermediate proteins controls the ability to use lactose an energy source. We consider Goodwin's ODE model of lac-operon dynamics, as well as a discrete DDE form of the same model, and reduce these models to a scalar distributed DDE. We note that the calculations shown here are easily generalisable to cyclic systems with $n \geq 4$.

We next consider a recent article studying white blood cell production \citep{Knauer2020}. The hematopoietic, or blood production, system has been modelled extensively, and these models often include explicit or implicit delays. As mentioned by Knauer et al. \citep{Knauer2020}, a compartmental system with linear feedback regulation implicitly includes a distributed delay, and the coupling of this delay with feedback is enough to produce oscillations. These oscillations are of particular interest in hematopoiesis due to the presence of so called ``dynamical diseases'' \citep{Mackey2020}. Here, we show that the Knauer et al. \citep{Knauer2020} model with maturation compartments and non-linear feedback also encodes a gamma type delay. 

\subsection{Models of lac-operon dynamics}
The lac-operon facilitates the use of lactose as a fuel source in certain types of bacteria and was one of the first genetic regulatory mechanisms to be understood. This regulatory mechanism is controlled by the presence of allolactose. In the presence of allolactose, mRNA transcription occurs and leads to the production of $\beta$-galactosidase, which converts allolactose to glucose. This conversion of allolactose eventually inhibits the production of mRNA and results in bistability in the operon. The lac-operon was one of the first genetic regulatory mechanisms to display such bistability.

Yildirim et al. \citep{Yildirim2004} proposed a reduced model of lac-operon dynamics to study the importance $\beta$-galactosidase on the bistability of the operon.  The structure of the reduced model proposed by Yildirim et al. \citep{Yildirim2004}  is similar to Goodwin's model of repressible dynamics \citep{Goodwin1965}. Before considering the Yildirim's DDE model of lac-operon dynamics, we study the simpler Goodwin \citep{Goodwin1965} model. Goodwin's model includes an metabolite controlled enzyme and intermediate stage and is known to produce oscillatory dynamics \citep{Goodwin1965}. 

Goodwin's model is a system of three differential equations modelling mRNA, $M(t)$; intermediate protein, $I(t)$; and effectors, $E(t)$ \citep{Goodwin1965}. The Goodwin model is a simple example of cyclic dynamics, where the production of one population is self regulating through the dynamics of the other two. By showing that the Goodwin model can be reduced to a scalar distributed DDE, we make this self-regulation explicit. The ODE model is
\begin{equation}
\left.
\begin{aligned}
\TimeDeriv M(t) & = F[E(t)] - \gamma_M M(t) \\
\TimeDeriv I(t) & = \alpha_I M(t)-\gamma_I I(t)\\
\TimeDeriv E(t) & = \alpha_E I(t) - \gamma_E E(t). 
\end{aligned}
\right \}
\label{Eq:LacOperonODE}
\end{equation}
The parameters $ \alpha_j$ and $ \gamma_j$ are positive real numbers for $j = M,I,E$ and represent the production and clearance of the $j$-th species, respectively. $F[E(t)]$ represents mRNA production driven by either an inducible or repressible operon, with the monotonicity of $F$ determining the type of feedback. As a first example of how to apply Theorem~\ref{Theorem:DistributedDDEReduction} in a cyclic feedback structure, we first reduce \eqref{Eq:LacOperonODE} to a distributed DDE where the effector $E(t)$ population is self-regulating.

Equation \eqref{Eq:LacOperonODE} is precisely in the same form as \eqref{Eq:GenericCharacteristicEquation} for specific choices of $f_I$ and $f_E$. We begin by considering the differential equation for $I(t)$. For
$ f_I(M(t)) = \alpha_I M(t)$,  as in \citep{Goodwin1965}, we have 
\begin{equation*}
\TimeDeriv I(t)  = f_I(M(t)) -\gamma_I I(t).
\end{equation*}
There is no other loss of intermediate proteins, so we can write
\begin{equation}
I(t) = \int_{-\infty}^t \alpha_I M(\phi) e^{-\gamma_I(t- \phi)} \d \phi = \int_{0}^{\infty} \alpha_I M(t-\phi)e^{-\gamma_I \phi} \d \phi .
\label{Eq:GoodwinIntermediateProteinExpression}
\end{equation}
Next, we consider the differential equation for the effector population, $E(t)$ with appearance rate
\begin{equation*}
f_E(I(t)) = \alpha_E I(t) = \alpha_E \int_{-\infty}^t \alpha_I M(\phi) e^{-\gamma_I(t- \phi)} \d \phi = \alpha_E \int_{0}^{\infty} \alpha_I M(t-\phi) e^{-\gamma_I \phi } \d \phi .
\end{equation*}
Once again, $\mu_E=\gamma_E,$ and we write
\begin{equation}
\begin{aligned}
E(t) & = \int_{0}^{\infty}  \alpha_E \underbrace{ \int_{-\infty}^{\theta} \alpha_I  M(\theta - \phi)e^{-\gamma_I(\theta-\phi)} \d \phi }_{I( \theta )} e^{-\gamma_E(t-\theta)} \d \theta .
\end{aligned}
\label{Eq:GoodwinEffectorExpression}
\end{equation}
Having expressed both $I(t)$ and $E(t)$ as functions of $M(t)$ in \eqref{Eq:GoodwinIntermediateProteinExpression} and \eqref{Eq:GoodwinEffectorExpression}, we can write the equivalent distributed DDE for the ODE model \eqref{Eq:LacOperonODE}
\begin{align}
\TimeDeriv M(t) & = F \left[ \int_{0}^{\infty}  \alpha_E  \int_{-\infty}^{\theta} \alpha_I  M(\theta - \phi)e^{-\gamma_I(\theta-\phi)} \d \phi  e^{-\gamma_E(t-\theta)} \d \theta  \right]  - \gamma_M M(t).
\label{Eq:LacOperonDDE}
\end{align}
There is no obvious ageing structure in the chain of enzyme, metabolite and intermediate protein. However, as mentioned, the cascade from metabolite to enzyme to intermediate protein defines a ``cyclic" model structure. In this sense, the metabolite controls it's own expression through \eqref{Eq:LacOperonDDE}. 
\subsubsection{Delayed lac-operon model}
Having shown how to reduce Goodwin's model of repressible dynamics to a scalar distributed DDE, we now consider the reduced Yildirim model of the delayed lac-operon \citep{Yildirim2004}. This model is given by three discrete DDEs
\begin{equation}
\left.
\begin{aligned}
\TimeDeriv M(t) & = F \left[ e^{-\nu_E \tau_M} E(t-\tau_M) \right] - \gamma_M M(t) \\
\TimeDeriv I(t) & = \alpha M(t-\tau_I)e^{-\nu_M \tau_I}- \gamma_I  I(t)\\
\TimeDeriv E(t) & = \alpha_E I(t)-\beta_E I(t)\frac{E(t)}{K_E+E(t)}   - \gamma_E E(t). 
\end{aligned}
\right \}
\label{Eq:YildrimModel}
\end{equation}
The model in equation~\eqref{Eq:YildrimModel} is slightly more complicated due to the presence of discrete delays and the non-linearity in the equation for $E(t)$. Due to the nonlinear Hill term in the differential equation for $E(t)$, we construct the cyclic structure in a different order than for the ODE lac-operon model. We begin with the equation for the dynamics of the metabolite $M(t)$. The metabolite is created as a function of the enzyme concentration at time $t-\tau_M$. Thus, 
\begin{equation*}
f_M \left( \int_0^{\infty} e^{\nu_E s }E(t-s)\delta(s-\tau_M)\d s \right) =  F \left[ e^{-\nu_E \tau_M} E(t-\tau_M) \right].
\end{equation*}
The metabolite is cleared at constant rate, so $\mu_M = \gamma_M$. Using Theorem \ref{Theorem:DistributedDDEReduction}, we obtain
\begin{align}  \label{Eq:MetaboliteDelayedLacOperon} 
M(t) &   = \int_0^{\infty}  F \left[ e^{-\nu_E \tau_M} E(t- \phi-\tau_M) \right]e^{-\gamma_M \phi} \d \phi .
\end{align}

Next, we consider the differential equation for the intermediate proteins $I(t)$. These intermediate proteins are created from the metabolite $M(t)$ following a delay $\tau_I$. The creation rate is
\begin{equation*}
f_I \left( \int_0^{\infty} e^{\nu_M s }I(t-s)\delta(s-\tau_I)\d s \right)  = \alpha_I M(t-\tau_I)e^{-\nu_M \tau_I},
\end{equation*}
There is no state dependent loss of the intermediate proteins, so $\mu_I =\gamma_I, $ so we find
\begin{align}\label{Eq:IntermediateProteinDelayedLacOperon} \notag
I(t)& = \int_{0}^{\infty} \alpha_I M(t-\theta-\tau_I) e^{-\nu_M \tau_I}e^{-\gamma_I(\theta)} \d \theta \\ 
 & = \int_{0}^{\infty} \alpha_I  \left[  \int_{0}^{\infty} F \left[ e^{-\nu_E \tau_M} E(t-\theta - \phi-\tau_I -\tau_M) \right] e^{-\gamma_M \phi}\d \phi \right]  e^{-\nu_M \tau_I} e^{-\gamma_I \theta} \d \theta.
\end{align}


This leaves the differential equation for the effector cells $E(t)$. Using \eqref{Eq:MetaboliteDelayedLacOperon} and \eqref{Eq:IntermediateProteinDelayedLacOperon}, we write the system \eqref{Eq:YildrimModel} as the following scalar distributed DDE
\begin{equation}
\left. 
\begin{aligned}
\TimeDeriv E(t) & = \int_{-\infty}^t \alpha_I  \left[  \int_{-\infty}^{\theta-\tau_I} F \left[ e^{-\nu_E \tau_M} E(\phi-\tau_M) \right] e^{-\gamma_M(\theta-\tau_I-\phi)} \d \phi \right] e^{-\nu_E \tau_I}e^{-\gamma_I (t-\theta)}  \d \theta \\
& {} \times \left[ \alpha_E -\beta_E \frac{E(t)}{K_E+E(t)}\right]  - \gamma_E E(t).
\end{aligned}
\right \}
\label{Eq:YildrimModelReduction}
\end{equation}


\subsubsection{Linearisation of the delayed lac-operon model}

Bistability in the lac-operon has been extensively studied and so we give an implicit expression for the equilibria solutions of \eqref{Eq:YildrimModelReduction}. Equation \eqref{Eq:YildrimModel} is a discrete DDE, so the canonical choice for the phase space is $\mathcal{C}( -\max[\tau_i,\tau_M],0)$ and equilibrium solutions are constant functions. Thus, we assume that $E(t)=E^*$ is a constant function and search for values of $E^*$ such that:
\begin{align*}
\gamma_E E^* =  \left( \alpha_E - \beta_E\frac{E^*}{K_E+E^*} \right) \left( \frac{\alpha_I}{\gamma_I}\right) \frac{F \left[ e^{-\nu_E \tau_M} E^* \right]}{\gamma_M} e^{-\nu_E \tau_I} .
\end{align*}

We use \eqref{Eq:IntermediateProteinDelayedLacOperon} to define 
\begin{align}
\bar{E}(t) & =  \int_{0}^{\infty} \alpha_I \left[  \int_{0}^{\infty} F \left[ e^{-\nu_E \tau_M} E(t-\theta - \phi-\tau_I -\tau_M) \right]e^{-\gamma_M \phi} \d \phi \right]  e^{-\nu_M \tau_I} e^{-\gamma_I \theta} \d \theta
\label{Eq:EbarYildrimModel}
\end{align}
Therefore, \eqref{Eq:YildrimModelReduction} becomes
\begin{equation*}
\TimeDeriv E(t) = \left( \alpha_E - \beta_E\frac{E(t)}{K_E+E(t)} \right) \bar{E}(t)-\gamma_E E(t),
\end{equation*}
and, after evaluating \eqref{Eq:EbarYildrimModel} at an equilibrium solution $E^*$, we must have
\begin{equation*}
\left( \alpha_E - \beta_E\frac{E^*}{K_E+E^*} \right)  \left( \frac{\alpha_I}{\gamma_I} \right)  \frac{F(e^{-\nu_E \tau_M}E^*)}{\gamma_M} e^{-\nu_M \tau_I} = \alpha_E \bar{E}^*  = \gamma_E E^*.
\end{equation*}

To study the behaviour of solutions near the equilibria $E^*$, we center the equilibria at the origin by considering $x(t) = E(t)-E^*.$  Now, as the decay rates are constant, it is straightforward to complete the linearisation by considering Taylor expansions. We present the details of this calculation rather than computing Fr\'{e}chet derivatives as in Section~\ref{Sec:ScalarDDEProperties}. 


It is natural to define 
\begin{equation}
\bar{x}(t) = \int_{0}^{\infty}  \alpha_I   \left[  \int_{0}^{\infty} \ F \left[ e^{-\nu_E \tau_M} (x(t-\theta - \phi-\tau_I -\tau_M) ) \right]e^{-\gamma_M \phi} \d \phi \right]  e^{-\nu_M \tau_I} e^{-\gamma_I \theta} \d \theta.
\label{Eq:XBarDefinition}
\end{equation}
and Taylor expanding $F(x(t)+E^*)$ about the equilibrium point $E^*$ gives 
\begin{align*}
\bar{x}(t) & = \left( \frac{\alpha_I}{\gamma_I} \right) \frac{F(e^{-\nu_E \tau_M}E^*)}{\gamma_M} e^{-\nu_M \tau_I} +  \int_{0}^{\infty}  \alpha_I  \left[  \int_{0}^{\infty}   \partial_xF(e^{-\nu_E \tau_M}E^*)\left[ x(t-\theta - \phi-\tau_I -\tau_M ) \right]  \right. \\ \notag
& \quad {} \times \left. e^{-\gamma_M \phi}\d \phi \right]  e^{-\nu_M \tau_I} e^{-\gamma_I \theta} \d \theta + \mathcal{O}(|x(t)|^2).
\end{align*}
Inserting the ansatz $x(t) = Ce^{\lambda t}$ and find
\begin{align*}
\bar{x}(t) & = \frac{F(e^{-\nu_E \tau_M}E^*)}{\gamma_M} e^{-\nu_M \tau_I} + Ce^{\lambda t} \frac{ \partial_xF(e^{-\nu_E \tau_M}E^*)   \alpha_I} e^{-\nu_M \tau_I}   e^{-\lambda(\tau_I+\tau_M)}   \\
& \quad {} \times \int_{0}^{\infty}  e^{-\lambda \theta} e^{-\gamma_I \theta} \d \theta  \int_{0}^{\infty}  e^{-\lambda \phi} e^{-\gamma_M \phi}\d \phi + \mathcal{O}(|x(t)|^2),
\end{align*}
so, after dropping non-linear terms, the differential equation for $x(t)$ is 
\begin{align} \notag  
\lambda x(t)  & =   \left[ \left(  \partial_xF(e^{-\nu_E \tau_M}E^*)e^{-\nu_M \tau_I} \gamma_M \right)  \mathcal{L}[\alpha_I](\gamma_I + \lambda)\mathcal{L}[ e^{-\lambda(\tau_I+\tau_M)}](\gamma_M+ \lambda) \right.  \\
& {} \quad \left. \times \left( \alpha_E -  \frac{\beta_E E^*}{K_E+E^*} \right)  - \left( \frac{\beta_E \bar{E}^* K_E}{(K_E+E^*)^2} + \gamma_E \right)  \right] x(t) .
\label{Eq:LinearisedDifferentialEquation}
\end{align}

After evaluating the Laplace transforms, dividing by $x(t) =Ce^{\lambda t}$, and using a common denominator, we obtain the characteristic equation corresponding to \eqref{Eq:LinearisedDifferentialEquation}
\begin{align*}
0 & = \left( \lambda+ \frac{\beta_E \bar{E}^* K_E}{(K_E+E^*)^2} + \gamma_E \right) (\lambda + \gamma_I)(\lambda+\gamma_M) \\
& \quad {} - \left( \alpha_E  - \beta_E\frac{E^*}{K_E+E^*} \right) \alpha_I \partial_xF(e^{-\nu_E \tau_M}E^*)e^{-\nu_M \tau_I}  e^{-\lambda(\tau_I+\tau_M)}.
\end{align*}
%
which is exactly the characteristic equation found by \citep{Yildirim2004} (after  undoing their nondimensionalization). Thus, we have shown how to reduce a system of three discrete DDEs to a scalar differential equation and have computed the characteristic equation without computing Jacobian matrices or determinants.

\subsection{Compartmental white blood cell model} \label{Sec:Hematopoiesis}

The human hematopoietic system is responsible for blood cells production and is tightly regulated by circulating cytokine concentrations. This cytokine control of blood cell production, maturation and release ensures that the hematopoietic system is able to respond to challenges such as infection, blood loss and hypoxemia. There has been extensive interest in mathematical modelling of the control mechanisms underlying the regulatory control of the hematopoietic system \citep{Pujo-Menjouet2016,Mackey1978}. In general, a circulating population of blood cells controls the production of precursors through a negative feedback loop mediated by cytokine signalling. In the absence of exogeneous cytokine administration, it is common to use a quasi-steady state approximation to discard a model for the cytokine signalling and simply use the circulating concentration of blood cells to control precursor production. Accordingly, these models typically exhibit the form of \eqref{Eq:GenericCyclicDifferentialEquation}. 

The production of neutrophils, the most common type of white blood cell in humans, has been extensively modelled over the past half century \citep{Pujo-Menjouet2016,Mackey1977,Rubinow1975,Dale2015}. Neutrophil precursors progress through a number of distinct proliferation and maturation stages before entering a reservoir of mature cells in the bone marrow and passing into circulation. It is common to model each of these stages separately, leading to a system of ODEs \citep{Rubinow1975,Roskos2006,Quartino2014,VonSchulthess1982}. Consequently, these models can be transformed to a distributed DDE through the LCT \citep{deSouza2017,Cassidy2018a}, where the distributed delay represents the time required for nascent neutrophil precursors to pass from the hematopoietic stem cell populations through proliferation and maturation before reaching circulation.

Marciniak-Czochra et al. \citep{Marciniak-Czochra2009} introduced a compartmental model of hematopoietic stem cell regeneration that has since been adapted to study bone marrow transplantation, resistance to therapy in leukemia, and other disorders of the hematopoietic system. Recently, the model was thoroughly analysed for two compartments in \citep{Getto2013}, who showed that the homeostatic equilibrium point is globally stable when it exists.

In a recent article, Knauer et al.\citep{Knauer2020} proposed a multi-compartment model for white blood cell production and demonstrated the existence of a super-critical Hopf bifurcation that leads to oscillatory circulating blood concentrations, similar to those observed in cyclic neutropenia \citep{Mir2020,Wright1981,Guerry1973,Dale1988}. Interestingly, the super-critical Hopf bifurcation and resulting periodic orbit results from the inclusion of a multi-stage maturation process \citep{Knauer2020}, and is not present in a similar model without the multiple maturation stages \citep{Getto2013}. This multi-stage maturation process results in the multi-compartment nature of the Knauer et al.\citep{Knauer2020} model, where each compartment corresponds to a distinct stage in the differentiation process. As the authors mention, these multi-compartment models have a long history in modelling cyclic neutropenia, and typically are structured to implicitly (or explicitly) induce a delay in the feedback. The Knauer et al. \citep{Knauer2020} model is the following three compartment model

\begin{equation}
\left.
\begin{aligned}
\TimeDeriv u_1(t) & = \left( 2\frac{a_1}{1+ku_3(t)}-1\right) p_1u_1(t) \\
\TimeDeriv u_2(t) & =  \left( 2\frac{a_2}{1+ku_3(t)}-1\right) p_2u_2(t) + 2 \left( 1- \frac{a_1}{1+ku_3(t)} \right)p_1u_1(t) \\
\TimeDeriv u_3(t) & = 2 \left( 1-\frac{a_2}{1+ku_3(t)} \right) p_2u_2(t) - d_3u_3(t).   
\end{aligned}
\right \}
\label{Eq:KnauerODEModel}
\end{equation}

Here, we show that the maturation stage in the compartmental model \eqref{Eq:KnauerODEModel} acts as to impose a distributed delay, and we reduce the system to a couple ODE and distributed DDE. This is a departure from earlier examples in which we completely reduced the system to a scalar distributed DDE. We note that the complete reduction is in fact possible for \eqref{Eq:KnauerODEModel}, but with an interesting complication: the differential equation for $u_1$ is linear in $u_1$ with $f_1(u_3(t)) = 0$. Consequently, the scalar DDE for $u_3$ explicitly depends on the initial condition $u_1(0)$. This explicit dependence on initial conditions is different than the preceding analysis and examples, and has a simple biological explanation: $u_1(0)$ represents the initial population of hematopoietic stem cells, from which white blood cells arise. 

In the previous examples, the feedback loop closes as the final stage drives control of the first. However, in \eqref{Eq:KnauerODEModel}, the hematopoietic stem cells $u_1$ begin the chain and are only  produced through self-renewal of the existing stem cell population. Thus, the circulating concentration of white blood cells will influence the growth or decay rate of the HSCs but cannot independently drive the production of new hematopoietic cells without HSC self-renewal. Therefore, we reduce \eqref{Eq:KnauerODEModel} to the a system for the HSC population and the circulating neutrophil concentration by replacing the intermediate compartment $u_2$ with a distributed delay, which leaves a system of equations for $u_1$ and $u_3$.

In \eqref{Eq:KnauerODEModel}, the effective proliferation rate of cells in compartment $i$ is given by $p_i$, with a fraction 
\begin{align*}
\left( 2\frac{a_i}{1+ku_3(t)}-1\right),
\end{align*}
of these cells self-renewing and remaining in the $i$-th compartment, while the remaining fraction
\begin{align*}
2\left( 1-\frac{a_i}{1+ku_3(t)} \right),
\end{align*}
progress to the subsequent compartment. Finally, mature cells are cleared from circulation linearly at a rate $d_3$. 

We begin with the differential equation for $u_2$
\begin{align*}
\TimeDeriv u_2(t) & = 2\left( 1- \frac{a_1}{1+k u_3(t)} \right) p_1 u_1(t) + \left( \frac{2a_2}{1+ku_3(t)} -1 \right) p_2u_2(t),
\end{align*}
and note that this differential equation has precisely the form of \eqref{Eq:GenericCharacteristicEquation} with
\begin{align*}
f_2 = 2\left( 1- \frac{a_1}{1+k u_3(t)} \right) p_1 u_1(t) \quad \textrm{and} \quad
\mu =   p_2 \left( \frac{2a_2}{1+ku_3(t)} -1 \right). 
\end{align*}
Thus, it follows that 
\begin{align}
u_2(t) = \int_0^{\infty} 2\left( 1- \frac{a_1}{1+k u_3(t-\sigma )} \right) p_1 u_1(t-\sigma) \exp \left[ p_2 \int_{t-\sigma}^t  \left( \frac{2a_2}{1+ku_3(x)} -1 \right) \d x \right] \d \sigma.
\label{Eq:KnauerU2Expression}
\end{align}

To facilitate the following computations, let
\begin{align*}
h_1(y) = 2p_1\left( 1- \frac{a_1}{1+ky} \right) \quad \textrm{and}
\quad h_2(y) = p_2\left( \frac{2a_2}{1+ky} -1 \right) 
\end{align*}
so that $h_2(u_3^*) = p_2 (a_2/a_1 -1)  < 0$ and $h_1(u_3^*) = p_1. $ Then, we can write \eqref{Eq:KnauerU2Expression} as
\begin{align*}
u_2(t) = \int_0^{\infty} h_1(u_3(t-\sigma))(u_1(t-\sigma))  \exp \left[   \int_{t-\sigma}^t  h_2(u_3(x)) \d x \right]\d \sigma ,
\end{align*}
and the Knauer et al. \citep{Knauer2020} model then reduces to 
\begin{equation}
\left. 
\begin{aligned}
\TimeDeriv u_1(t) & = \left( 2\frac{a_1}{1+ku_3(t)}-1\right) p_1u_1(t) \\
\TimeDeriv u_3(t) & = \left( \int_0^{\infty} h_1(u_3(t-\sigma))(u_1(t-\sigma))  \exp \left[   \int_{t-\sigma}^t  h_2(u_3(x)) \d x \right]\d \sigma \right) \\ 
& {} \quad  \times 2 p_2 \left( 1-\frac{a_2}{1+ku_3(t)} \right)   - d_3u_3(t).   
\end{aligned}
\right \}
\label{Eq:Knauer2CompartmentDDEModel}
\end{equation} 

In the preceding calculation, we have implicitly assumed that $u_1(0) \neq 0$. Now, if $u_1(0) =0$, then $u_1(t) = 0$ for all $t>0$ and the 3 compartment \eqref{Eq:KnauerODEModel} becomes

\begin{equation}
\left.
\begin{aligned}
\TimeDeriv u_2(t) & =  \left( 2\frac{a_2}{1+ku_3(t)}-1\right) p_2u_2(t) \\
\TimeDeriv u_3(t) & = 2 \left( 1-\frac{a_2}{1+ku_3(t)} \right) p_2u_2(t) - d_3u_3(t).   
\end{aligned}
\right \}
\label{Eq:SingularKnauerODEModel}
\end{equation} 
Then, the preceding discussion regarding the biological interpretation of $u_1(0)$ for \eqref{Eq:KnauerODEModel} can be repeated verbatim for \eqref{Eq:SingularKnauerODEModel} but now with $u_2(0)$.  

\subsubsection{Equilibria and linearisation}

In this form, the equilibria solutions are the constant functions $(u_1(t),u_3(t)) = (u_1^*,u_3^*)$ such that the right hand side of \eqref{Eq:Knauer2CompartmentDDEModel} is zero. Immediately, we see that
\begin{equation*}
u_1^* = 0 \quad \textrm{or} \quad u_3^* = \frac{2a_1-1}{k}.
\end{equation*}

Using the equilibrium value of $u_3^*$, the non-zero equilibria value of $u_1^*$ is given by
\begin{align*}
d_3 u_3^* & =  4 p_1p_2 \left( 1- \frac{a_1}{1+k u_3^*} \right) \left( 1-\frac{a_2}{1+ku_3^*} \right)  u_1^*  \int_0^{\infty} \exp \left[ p_2  \left( \frac{2a_2}{1+ku_3^*} -1 \right)\sigma  \right] \d \sigma \\
& = 2 p_1 \left( 1-\frac{a_2}{2a_1} \right) u_1^* \int_0^{\infty} p_2  \exp \left[- p_2  \left( 1 -  \frac{a_2}{a_1)} \right)\sigma  \right] \d \sigma = \frac{ p_1 \left( 2-\frac{a_2}{a_1} \right) u_1^* }{1 -  \frac{a_2}{a_1} },
\end{align*} 
which is precisely the value found by \citep{Knauer2020} and only exists if $a_2< a_1$.

Now, to linearise about the equilibrium point, consider $z(t) = u(t)-u^*$ and, for $F$ given by the right hand side of  \eqref{Eq:Knauer2CompartmentDDEModel}, we obtain the differential equation for $z$
\begin{align*}
\TimeDeriv z(t) = \TimeDeriv u(t) = F(u^*+z(t))
\end{align*} 
We begin with the computation of the linearisation of the delayed term 
\begin{align*}
\int_0^{\infty} h_1(u_3^*+ z_3(t-\sigma))(u_1^*+z_1(t-\sigma))  \exp \left[   \int_{t-\sigma}^t  h_2(u_3^*+z_3(x)) \d x \right].
\end{align*}
Taylor expanding the above expression in $z_1$ and $z_3$ gives 
\begin{align*}
& \int_0^{\infty} [h_1(u_3^*)+h_1'(u_3^*)z_3(t-\sigma)](u_1^*+z_1(t-\sigma)) e^{ h_2(u_3^*)\sigma } \left( 1+ \int_{t-\sigma}^t h'_2(u_3^*)z_3(x)\d x  \right) \d \sigma + \mathcal{O}(z^2) \\
& {} \quad = \frac{h_1(u_3^*)u_1^* }{h_2(u_3^*)} + \int_0^{\infty} [h_1'(u_3^*)u_1^*z_3(t-\sigma) + h_1(u_3^*)z_1(t-\sigma) ]e^{h_2(u_3^*)\sigma} \d \sigma \\
& {} \quad  + \int_0^{\infty} h_1(u_3^*)u_1^* e^{h_2(u_3^*)\sigma} \left[ \int_{t-\sigma}^t h'_2(u_3^*) z_3(x) \d x \right]  \d \sigma + \mathcal{O}(z^2).
\end{align*}

We note that the Fr\'{e}chet derivative of a linear operator is the operator itself. Thus, to simplify notation, we discard the non-linear terms and insert the ansatz $z(t) = ce^{\lambda t}$ to find
\begin{align} \notag
& \frac{h_1(u_3^*)u_1^* }{h_2(u_3^*)} + \int_0^{\infty} [h_1'(u_3^*)u_1^*z_3(t-\sigma) + h_1(u_3^*)z_1(t-\sigma) ]e^{h_2(u_3^*)\sigma} \d \sigma \\ \notag
& {} \quad  + \int_0^{\infty} h_1(u_3^*)u_1^* e^{h_2(u_3^*)\sigma} \left[ \int_{t-\sigma}^t h'_2(u_3^*) z_3(x) \d x \right]  \d \sigma + \mathcal{O}(z^2) \\ \notag
& {} \quad = \frac{h_1(u_3^*)u_1^* }{h_2(u_3^*)} + \mathcal{L}[h_1'(u_3^*)u_1^*](\lambda+ h_2(u_3^*))z_3(t)  + \mathcal{L}[ h_1(u_3^*)](\lambda+ h_2(u_3^*))z_1(t) \\
& {} \quad  + \int_0^{\infty} h_1(u_3^*)u_1^* e^{h_2(u_3^*)\sigma} \left[ \int_{t-\sigma}^t h'_2(u_3^*) z_3(x) \d x \right]  \d \sigma.
\label{Eq:KnauerExpansionIntegral}
\end{align}
Using the ansatz $z_3(t) = ce^{\lambda t }$, we can easily calculate
\begin{equation*}
\int_{t-\sigma}^t h'_2(u_3^*) z_3(x) \d x = h'_2(u_3^*) \left( \frac{z_3(t)-z_3(t-\sigma)}{\lambda}\right) . 
\end{equation*}
Inserting this into \eqref{Eq:KnauerExpansionIntegral} then gives
\begin{align*}
& \int_0^{\infty} h_1(u_3^*)u_1^* e^{h_2(u_3^*)\sigma} \left[ \int_{t-\sigma}^t h'_2(u_3^*) z_3(x) \d x \right]  \d \sigma  = \left( \frac{h_1(u_3^*)u_1^*  h'_2(u_3^*)}{\lambda} \right) \left( \frac{1}{h_2(u_3^*)} - \frac{1}{h_2(u_3^*)+\lambda} \right)
\end{align*}
which, after using a common denominator and simplifying, gives
\begin{align*}
\int_0^{\infty} h_1(u_3^*)u_1^* e^{h_2(u_3^*)\sigma} \left[ \int_{t-\sigma}^t h'_2(u_3^*) z_3(x) \d x \right]  \d \sigma   & = \frac{ h_1(u_3^*)u_1^*  h'_2(u_3^*)}{h_2(u_3^*)}  \left(  \frac{1}{h_2(u_3^*)+\lambda} \right) \\
& = {} \mathcal{L}\left[ \frac{ h_1(u_3^*)u_1^*  h'_2(u_3^*)}{h_2(u_3^*)} \right] (h_2(u_3^*)+\lambda).
\end{align*}

Thus, the linear differential equation for $z_3(t)$ is
\begin{align*}
\TimeDeriv z_3(t) & = -d_3(z_3(t)+u_3^*) + \left[ 2 p_2 \left( 1-\frac{a_2}{1+ku_3^*} \right) + 2p_2  \frac{ka_2}{(2a_1)^2}z_3(t) + \mathcal{O}(z^2)  \right] \\
& \quad {} \times \left( \frac{h_1(u_3^*)u_1^*}{p_2(1-a_2/a_1)} + \mathcal{L} \left[ h_1'(u_3^*)u_1^*+\frac{ h_1(u_3^*)u_1^*  h'_2(u_3^*)}{h_2(u_3^*)} \right](\lambda + p_2(1-a_2/a_1) )z_3(t) \right. \\
& \quad \left.   + \mathcal{L}[h_1(u^*)](\lambda + p_2(1-a_2/a_1) )z_1(t) + \mathcal{O}(z^2)    \right) \\ 
& {} = \left( -d_3 +  2p_2  \frac{ka_2}{(2a_1)^2} \frac{h_1(u_3^*)u_1^*}{p_2(1-a_2/a_1)} \right.  \\
& {}  \quad \left. +  2 p_2 \left( 1-\frac{a_2}{1+ku_3^*} \right) \mathcal{L} \left[ h_1'(u_3^*)u_1^*+\frac{ h_1(u_3^*)u_1^*  h'_2(u_3^*)}{h_2(u_3^*)} \right](\lambda + p_2(1-a_2/a_1) ) \right) z_3(t) \\
& \quad {} + \left( 2 p_2 \left( 1-\frac{a_2}{1+ku_3^*} \right) \mathcal{L}[h_1(u^*)](\lambda + p_2(1-a_2/a_1) \right) z_1(t).
\end{align*}

From which we get the linearised differential equation for $z(t)$
\begin{align*}
\TimeDeriv z(t) = A z(t)
\end{align*}
 where the linearisation matrix $A$ is given by 
\begin{equation*}
A(\lambda) = \left[ 
\begin{array}{cc}
0 & \left( 1- \frac{1}{2a_1}\right)\frac{d_3}{2-a_2/a_1} \left( 1- a_2/a_1\right) \\
  p_2 \left( 2 -\frac{a_2}{a_1} \right) \mathcal{L}[h_1(u^*)]( \lambda + p_2(1-a_2/a_1) )  & d_3 \left[ \left(1-\frac{1}{2a_1}\right)\frac{a_2}{a_1}\frac{1}{2-\frac{a_2}{a_1} }-1\right]  + A_{22}(\lambda)\\
\end{array}
\right]
\end{equation*}
where
\begin{align*}
A_{22}(\lambda) &  =  2 p_2 \left( 1-\frac{a_2}{1+ku_3^*} \right)  \mathcal{L} \left[ h_1'(u_3^*)u_1^*+\frac{ h_1(u_3^*)u_1^*  h'_2(u_3^*)}{h_2(u_3^*)} \right](\lambda + p_2(1-a_2/a_1) ).
\end{align*}

Following \citep{Knauer2020} and rescaling time by $\hat{t} = tp_1,$ we have $h_1(u_3^*) = 1$, and we simplify
\begin{align*}
A_{22}(\lambda)  = p_2d_3\left(1-2\frac{a_2}{a_1} \right) \left(1- \frac{1}{2a_1} \right).
\end{align*}


Then, computing $\det\left[ \lambda I - A \right]$ gives the same characteristic equation as was found in  \citep{Knauer2020}
\begin{align*}
0 & = \lambda^3 + \left[ \left( 1-\frac{a_2}{a_1} \right) p_2 +  \left( 1-\frac{a_2}{a_1} \right) \left( 1-\frac{1}{2a_1} \right) \frac{1}{2-\frac{a_2}{a_1} }  \right] \lambda^2  \\
& {} + \left[ \left( 1-\frac{a_2}{a_1} \right) \left( 1 - \frac{a_2}{a_1} \right) \left( 1-\frac{1}{2a_1}\right) \frac{1}{2-\frac{a_2}{a_1}}  - \left( 1-\frac{1}{2a_1} \right) \left( 1-2\frac{a_2}{a_1} \right) \right] d_3 p_2 \lambda \\
& {} +\left( 1-\frac{1}{2a_1} \right) \left( 1-2\frac{a_2}{a_1} \right) d_3 p_2. 
\end{align*}

\subsubsection{Biological Interpretation}

Oscillations in mathematical models of hematopoiesis have been extensively studied, with cyclic neutropenia being a canonical example of a ``dynamical disease.'' Mathematical models of these diseases often share a recipe of delayed feedback leading to oscillations. Here, we show that the Knauer et al.\citep{Knauer2020} model also shares this framework. This is particularly interesting, as the in-depth anaylsis of Getto et al. \citep{Getto2013} demonstrates that the Knauer et al. \citep{Knauer2020} model without the maturation compartment \textit{cannot} produce oscillatory solutions. Conversely, the multistage compartment model in \eqref{Eq:KnauerODEModel} undergoes a Hopf bifurcation and produces solutions that compare favourably with observed data from patients with cyclic neutropenia. Thus, it appears that the inclusion of a delay between signal and response in the feedback loop is necessary, at least in this model formulation, to recapture the oscillatory dynamics observed in the hematopoietic system.

\section{Conclusion}

In this work, we have formalized the relationship between cyclic differential equations and distributed DDEs. This relationship is well-known in the case of transit compartment models as  the \textit{linear chain technique,} and has been shown to lead to state dependent distributed DDEs in the variable transit rate case \citep{Cassidy2018a}. However, both of these equivalences require linear transit between compartments. At the heart of the LCT is the ability to write down a closed form integral solution of the transit compartment model. Here, we use the same idea in a more general setting to establish the equivalence between more general cyclic differential equations and distributed DDEs by writing an integral form solution of the transit compartments. In essence, we demonstrate how sequentially solving the transit compartment system naturally leads to a scalar distributed DDE.

The reduction of a generic cyclic model to a scalar distributed DDE has a number of advantages. Mathematically, determining the existence of equilibria in $n$ dimensional systems typically requires solving $n$ simultaneous equations, and it is, in general, difficult to determine if the equilibrium point is unique. Conversely, both de Souza et al. and Cassidy et al.  demonstrate that the distributed DDE formulation of transit compartment models can be more tractable to analytical techniques \citep{deSouza2017,Cassidy2018a}.  For example, once an equilibrium point has been found, studying the local stability of an equilibrium involves the calculation of the eigenvalues of the $n \times n $ Jacobian matrix. Consequently, if modelling biological data indicates the need for the inclusion of an additional intermediate modelling stage, it is necessary to effectively recalculate the now $(n+1) \times (n+1)$ Jacobian matrix and it's eigenvalues from scratch.  Conversely, when working with the equivalent scalar distributed DDE, we can use tools from single variable calculus such as the intermediate value theorem to determine the existence and uniqueness of equilibria. Further, studying the local stability of these equilibria corresponds to calculating a single Fr\'{e}chet derivative. As we have shown, this calculation replaces the calculation of the determinant of the $n \times n$ Jacobian matrix with the chain rule of Fr\'{e}chet derivatives, and is much more amendable to the inclusion of new modelling stages.

Biologically, the scalar distributed DDE explicitly identifies delays between signal and response that are otherwise hidden in the equivalent cyclic system. Moreover, each intermediate stage represents another quantity that should be compared to data when validating a mathematical model. However, these intermediate stages are either often difficult to measure or do not represent specific physiological compartments. To emphasize this point, we considered two examples that represent biological systems without obvious delays, and showed that identifying the otherwise hidden delays can suggest necessary model ingredients to recapture biological phenomena, as in the Section~\ref{Sec:Hematopoiesis}. Conversely, when considering the equivalent scalar distributed DDE, the model output may be easier to compare against biological data. In a related point, using the scalar distributed DDE formulation can alleviate non-biological modelling assumptions. For example, using a transit compartment ODE model to replace a distributed DDE imposes a non-biological constraint on the delayed process. Namely, imposing that the delayed process be Erlang distributed constrains one of the two parameters of the gamma distribution. As the mean and variance of a delayed process precisely determine the shape and scale parameters of the gamma distribution, imposing that the shape parameter is an integer leads to an over determined system for the remaining scale parameter. For example, when modelling the duration of the cell cycle using an Erlang distributed DDE, modellers can capture the mean or the variance of the delayed process, but not generally both \citep{Cassidy2019}. This limitation can be alleviated when using the more general distributed DDE.

In summary, we formalize the equivalence between cyclic systems of differential equations with delay and scalar distributed DDEs. However, the distributed DDE formulation of cyclic models has some limitations. The most striking of these is the lack of established numerical techniques for the simulation and bifurcation analysis of infinite delay models, although recent work has alleviated this limitation somewhat \citep{Gyllenberg2018,Diekmann2020c} Nevertheless, the equivalence established in this work allows researchers to study the the mathematical model in whichever form is most convenient, and may elucidate otherwise hidden delayed processes. 


\section*{Acknowledgments}
I am grateful to Tony Humphries, Morgan Craig, and Michael C. Mackey for helping shape this manuscript. This work was partially funded by a NSERC PGS-D award. Portions of this work were performed under the auspices of the U.S. Department of Energy under contract 89233218CNA000001 and funded by NIH grants R01-AI116868 and R01-OD011095.


\end{document}